\documentclass[12pt]{amsart}
\newtheorem{theorem}{Theorem}[section]

\newtheorem{problem}[theorem]{Problem}

\newtheorem{proposition}[theorem]{Proposition}
\usepackage{graphicx}
\begin{document}
\author{Gerardo Hern\'andez-del-Valle}
\title[On Schr\"odinger's equation, 3-dimensional bessel bridges]{On Schr\"odinger's equation, 
 3-dimensional Bessel bridges, and passage time problems}
\address{Statistics Deparment, Columbia University\\ Mail Code
4403, New York, N.Y.} \email{gerardo@stat.columbia.edu}
\date{May 2010}

\maketitle
\begin{abstract}
In this work we relate the density of the first-passage time of a Wiener process to a moving boundary with the three dimensional Bessel bridge process and a solution of the heat equation with a moving boundary. We provide bounds.
\end{abstract}

\section{The problem}\label{sec1}
\begin{problem}
The main motivation of this work is to find the density of the first time $T$, that a one-dimensional, standard Brownian motion $B$, reaches the moving boundary $f$:
\begin{equation}\label{stop}
T:=\inf\left\{t\geq 0|B_t=f(t)\right\}
\end{equation}
where $f(t):=a+\int_0^tf'(u)du$,   $f''(t)>0$, and $\int_0^t(f'(u))^2du<\infty$ for all $t>0$.
\end{problem}
 Alternatively, assuming the previous assumptions:

\begin{proposition} Given that
\begin{equation}\label{level}
h(s,a):=\frac{|a|}{\sqrt{2\pi s^3}}\exp\left\{-\frac{a^2}{2s}\right\}\qquad s\geq 0,\enskip a\in\mathbb{R}
\end{equation}
is the density of the first time that $B$ reaches the fixed level $a$, the process $\tilde{X}$ is a 3-dimensional Bessel bridge, which has the following dynamics:
\begin{equation*}
d\tilde{X}_t=dW_t+\left(\frac{1}{\tilde{X}_t}-\frac{\tilde{X}_t}{s-t}\right)dt,\qquad \tilde{X}_0=a,\enskip t\in[0,s).
\end{equation*}
(Where $W$ is a Wiener process.)
Then, the distribution of $T$ equals
\begin{eqnarray}
\nonumber\mathbb{P}(T<t)&=&
\int_0^t\tilde{\mathbb{E}}\left[\exp\left\{-\int_0^sf''(u)\tilde{X}_udu\right\}\right]\\
\label{cota}&&\quad\times \exp\left\{-\frac{1}{2}\int_0^s(f'(u))^2du-f'(0)a\right\}h(s,a)ds
\end{eqnarray}
\end{proposition}
\begin{proof} From Girsanov's theorem, the fact that the boundary $f$ is twice continuously differentiable, the following Radon-Nikodym derivative:
\begin{eqnarray*}
\frac{d\mathbb{P}}{d\tilde{\mathbb{P}}}(t)&:=&\exp\left\{-\int_0^tf'(s)d\tilde{B}_s-\frac{1}{2}\int_0^t(f'(s))^2ds\right\}\\
&:=&\exp\left\{-f'(t)\tilde{B}_t+\int_0^tf''(s)\tilde{B}_sds-\frac{1}{2}\int_0^t(f'(s))^2ds\right\},
\end{eqnarray*}
is indeed a {\it martingale\/} [Novikov's condition] and induces the following relationship:
\begin{eqnarray*}
\mathbb{P}(T\leq
t)&:=&\tilde{\mathbb{E}}\left[\frac{d\mathbb{P}}{d\tilde{\mathbb{P}}}(t)\mathbb{I}_{(T\leq
t)}\right]\\
&=&\tilde{\mathbb{E}}\left[\frac{d\mathbb{P}}{d\tilde{\mathbb{P}}}(T)\mathbb{I}_{(T\leq
t)}\right]\\
&=&\int_0^t\tilde{\mathbb{E}}\left[\frac{d\mathbb{P}}{d\tilde{\mathbb{P}}}(T)\Big{|}T=s\right]h(s,a)ds,
\end{eqnarray*}
where the second equality follows from the optional sampling theorem, and the third   from conditioning with respect to $T$ under $\tilde{P}$, and $h$ is defined in (\ref{level}).

In Revuz and Yor (1999) (Chapter 11), it is shown that calculating:
\begin{equation}\label{expected}
\tilde{\mathbb{E}}\left[\frac{d\mathbb{P}}{d\tilde{\mathbb{P}}}(T)\Big{|}T=s\right]
\end{equation}
is equivalent to finding the expected value of a functional of a 3-dimensional Bessel bridge $\tilde{X}$, which at time 0 starts at $a$ and at time $s$ equals 0. Indeed,
\begin{eqnarray}
&&\nonumber\tilde{\mathbb{E}}\left[\frac{d\mathbb{P}}{d\tilde{\mathbb{P}}}(T)\Big{|}T=s\right]\\
&&\nonumber\quad=
\tilde{\mathbb{E}}\Bigg{[}\exp\Bigg{\{}-f'(T)a+\int_0^Tf''(u)\tilde{B}_udu\\
&&\nonumber\quad\qquad -\frac{1}{2}\int_0^T(f'(u))^2du\Bigg{\}}\Big{|}T=s\Bigg{]}\\
&&\nonumber\quad=
\tilde{\mathbb{E}}\left[\exp\left\{-f'(s)a+\int_0^sf''(u)(a-\tilde{X}_u)du-\frac{1}{2}\int_0^s(f'(u))^2du\right\}\right]\\
&&\nonumber\quad =
\tilde{\mathbb{E}}\Bigg{[}\exp\Bigg{\{}-f'(s)a+(f'(s)-f'(0))a-\int_0^sf''(u)\tilde{X}_udu\\
&&\nonumber\quad\qquad-\frac{1}{2}\int_0^s(f'(u))^2du\Bigg{\} }\Bigg{]}\\
&&\label{problem}\quad =
\tilde{\mathbb{E}}\left[\exp\left\{-\int_0^sf''(u)\tilde{X}_udu\right\}\right]\\
&&\nonumber\qquad\quad\times\exp\left\{-\frac{1}{2}\int_0^s(f'(u))^2du-f'(0)a\right\}
\end{eqnarray}
 thus the process $a-\tilde{X}$  equals 0 at $t=0$, and at $s$ reaches level $a$ for the first time, as required [see Figures \ref{f1}--\ref{f3}].

\end{proof}

From equation (\ref{problem}) it is clear that our next goal is to compute the following expected value:
\begin{equation}\label{ex}
\tilde{\mathbb{E}}\left[\exp\left\{-\int_0^sf''(u)\tilde{X}_udu\right\}\right].
\end{equation}

\begin{theorem}
Suppose that $v(t,a):[0,s]\times \mathbb{R}^+\to\mathbb{R}^+$ is continuous, is of class $\mathbb{C}^{1,2}$ and satisfies the Cauchy problem
\begin{eqnarray}
\label{cau}-\frac{\partial v}{\partial t}+f''(t)av&=&\frac{1}{2}\frac{\partial^2 v}{\partial a^2}+\left(\frac{1}{a}-\frac{a}{s-t}\right)\frac{\partial v}{\partial a}\qquad[0,s)\times\mathbb{R}^+,\\
\nonumber v(s,a)&=&1,\qquad a\in\mathbb{R}^+,
\end{eqnarray}
as well as
\begin{eqnarray*}
0\leq v(t,a)\leq 1\qquad \forall\enskip t,a\in\mathbb{R}^+
\end{eqnarray*}
Then $v(t,a)$ admits the stochastic representation
\begin{eqnarray*}
v(t,a)=\mathbb{E}^{t,a}\left[\exp\left\{-\int_t^sf''(u)\tilde{X}_udu\right\}\right]
\end{eqnarray*}
\end{theorem}

\begin{proof} We proceed as in the proof of Theorem 5.7.6, pp. 366--367, Karatzas \& Shreve (1991).
Applying Ito's rule to the process 
$$v(y,X_y)\exp\left\{-\int^t_yf''(u)\tilde{X}_udu\right\},$$  $y\in[t,s]$, and obtain, with $\tau_n:=\inf\{t\leq y\leq s| \tilde{X}_y\geq n\}$,
\begin{eqnarray*}
v(t,a)&=&\tilde{\mathbb{E}}^{t,a}\left[\exp\left\{-\int_t^sf''(u)\tilde{X}_udu\right\}\mathbb{I}_{(\tau_n>s)}\right]\\
&&+\tilde{\mathbb{E}}^{t,a}\left[v(\tau_n,X_{\tau_n})\exp\left\{-\int_t^{\tau_n}f''(u)\tilde{X}_udu\right\}\mathbb{I}_{(\tau_n\leq s)}\right]
\end{eqnarray*}
The second term converges to zero as $n\to\infty$, since
\begin{eqnarray*}
&&\tilde{\mathbb{E}}^{t,a}\left[v(\tau_n,X_{\tau_n})\exp\left\{-\int_t^{\tau_n}f''(u)\tilde{X}_udu\right\}\mathbb{I}_{(\tau_n\leq s)}\right]\\&&\qquad\qquad\leq \tilde{\mathbb{E}}^{t,x}\left[v(\tau_n,\tilde{X}_{\tau_n})\mathbb{I}_{(\tau_n\leq s)}\right]\\
&&\qquad\qquad\leq\tilde{\mathbb{P}}^{t,x}(\tau_n\leq s)\\
&&\qquad\qquad=\tilde{\mathbb{P}}^{t,x}\left(\max\limits_{t\leq\theta\leq s}\tilde{X}_\theta\geq n\right)\\
&&\qquad\qquad\leq \frac{\tilde{\mathbb{E}}^{t,x}\left[\max\limits_{t\leq\theta\leq s}\tilde{X}_\theta^{2m}\right]}{n^{2m}}
\end{eqnarray*}
[see Pitman and Yor (1998) for the moments of the running maximum of $\tilde{X}$]. Finally, the first term converges to
\begin{eqnarray*}
\tilde{\mathbb{E}}^{t,a}\left[\exp\left\{-\int_t^sf''(u)\tilde{X}_udu\right\}\right]
\end{eqnarray*}
either by the dominated or by the monotone convergence theorem.
\end{proof}

\vfill\eject 
\begin{figure}
\begin{center}
\includegraphics[height=2.1in,width=2.5in]{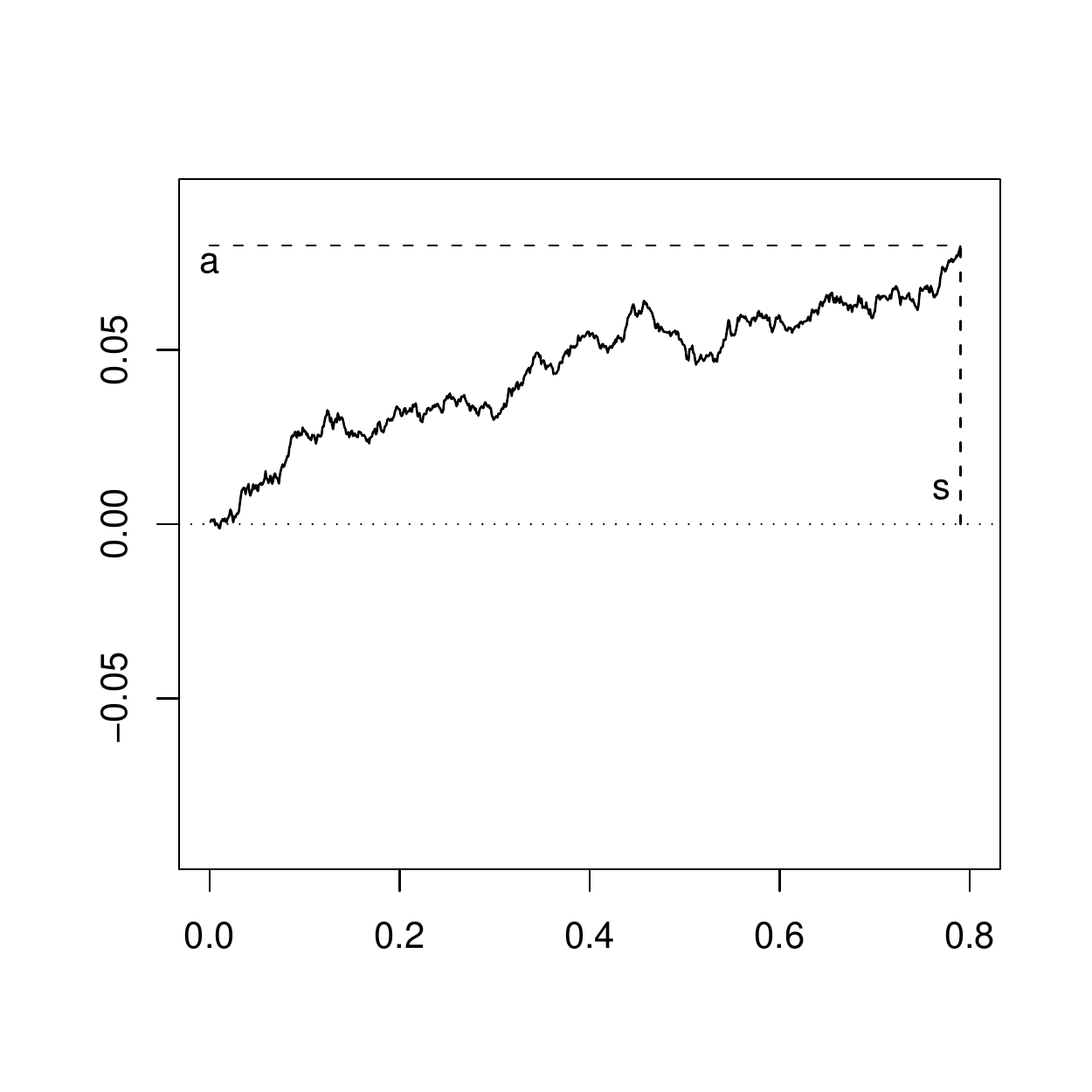}
\caption{(a) Brownian motion conditioned to hit level $a$ for the first time at time $s$.}\label{f1}
\includegraphics[height=2.1in,width=2.5in,scale=0.5]{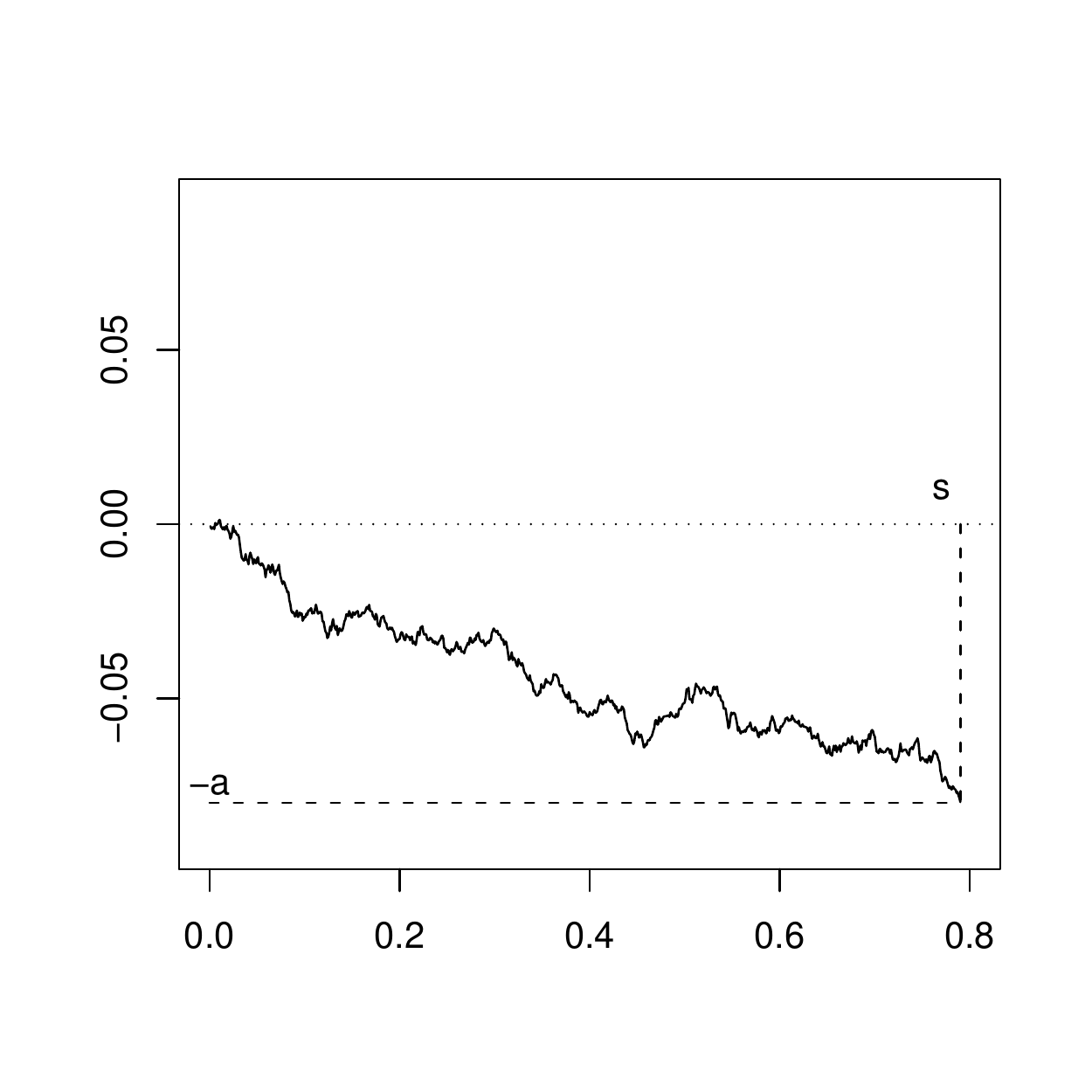}
\caption{(b) Reflection of the conditioned process.}\label{f2}
\includegraphics[height=2.1in,width=2.5in,scale=0.5]{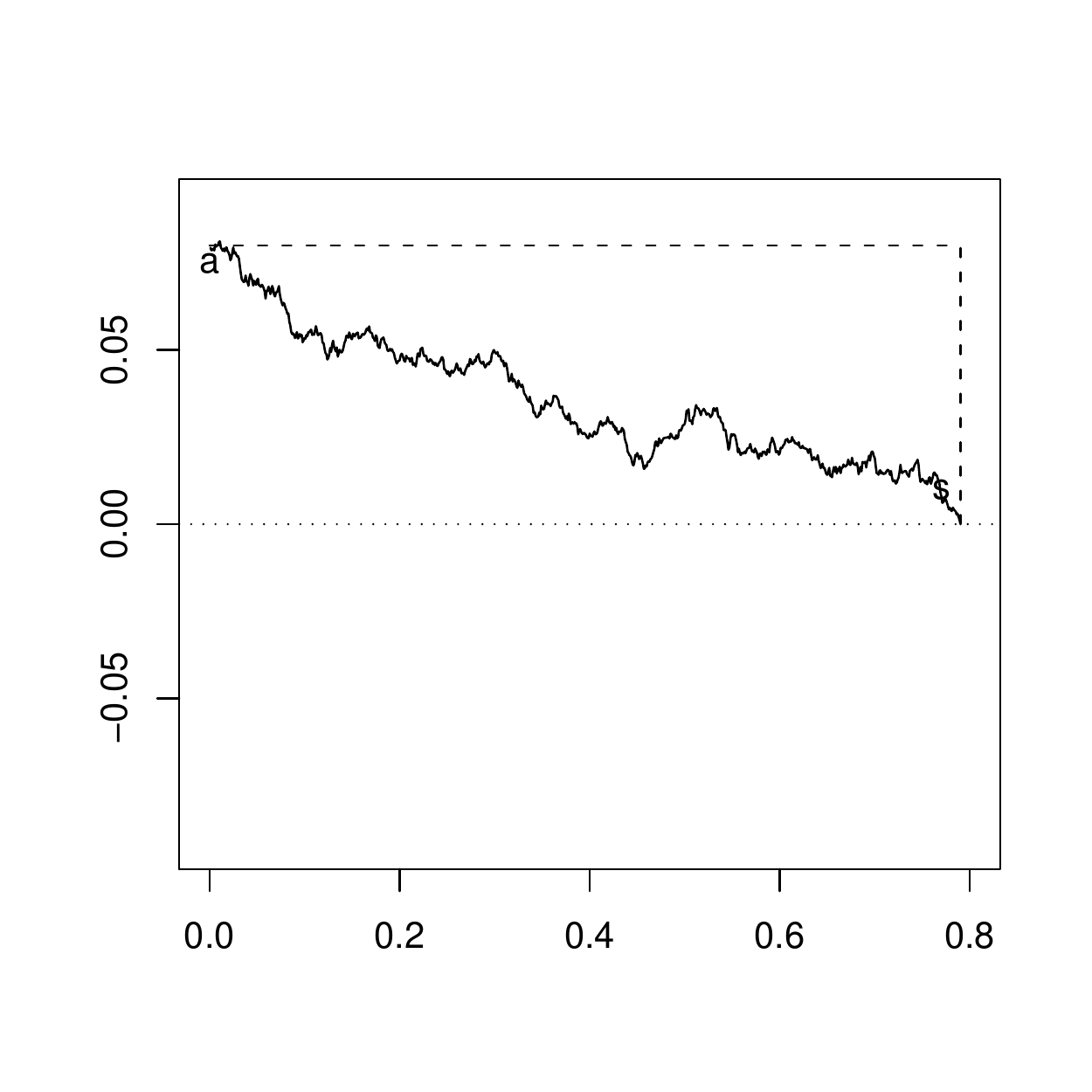}
\caption{(c) Three-dimesional Bessel bridge $\tilde{X}$}\label{f3}
\end{center}
\end{figure}

\vfill\eject

\section{Solutions of equation (\ref{cau})}
\begin{proposition} Solutions to (\ref{cau}) are of the following form
$$
v(t,a)=\frac{w(t,a)}{h(s-t,a)}
$$
where
\begin{eqnarray}\label{schro}
-w_t(t,a)+f''(t)aw(t,a)=\frac{1}{2}w_{aa}(t,a).
\end{eqnarray}
and $h(s,a)$ as in (\ref{level}).
\end{proposition}
\begin{proof} 
Setting
$$u(t,a)=1/h(s-t,a)$$
and $v(t,a)=u(t,a)w(t,a)$, we have that
\begin{eqnarray*}
u_t(t,a)&=&\left[\frac{a^2}{2(s-t)^2}-\frac{3}{2(s-t)}\right]u(t,a)\\
 u_a(t,a)&=&-\left[\frac{1}{a}-\frac{a}{s-t}\right]u(t,a)\\
u_{aa}(t,a)&=&\left[\frac{2}{a^2}+\frac{a^2}{(s-t)^2}-\frac{1}{s-t}\right]u(t,a)
\end{eqnarray*}
and 
\begin{eqnarray*}
v_t&=&u_tw+uw_t\qquad v_x=u_xw+uw_x\\
v_{xx}&=&u_{xx}w+2u_xw_w+uw_{xx}.
\end{eqnarray*}
Hence from (\ref{cau}) and (\ref{schro}) it follows that
\begin{eqnarray*}
\left[-u_t-\frac{1}{2}u_{aa}-\left(\frac{1}{a}-\frac{a}{s-t}\right)u_a\right]w=\left[u_a+\left(\frac{1}{a}-\frac{a}{s-t}\right)u\right]w_a,
\end{eqnarray*}
as claimed.
\end{proof}
\begin{theorem}
Solutions to (\ref{schro}) are given by
\begin{eqnarray}
\label{sch1}&&\\
\nonumber w(t,a)&=&e^{\frac{1}{2}\int_t^s(f'(u))^2du+af'(t)}\frac{1}{2\pi}\int_{-\infty}^\infty \Pi(y)e^{-\frac{1}{2}y^2(s-t)+iy(a+\int_t^sf'(u)du)}dy\\
\nonumber&=&e^{\frac{1}{2}\int_t^s(f'(u))^2du+af'(t)}\omega(s-t,a+\int_t^sf'(u)du)
\end{eqnarray}
where $\omega$ is a solution to the heat equation.
\end{theorem}

\begin{proof} Let
\begin{eqnarray}\label{fou}
\hat{w}(t,\lambda):=\int\limits_{-\infty}^\infty e^{-i\lambda a}w(t,a)da.
\end{eqnarray}
Applying the Fourier transform to (\ref{schro}) we have
\begin{eqnarray*}
-\hat{w}_t(t,\lambda)+if''(t)\hat{w}_\lambda(t,\lambda)+\frac{1}{2}\lambda^2\hat{w}(t,\lambda)=0\qquad i:=\sqrt{-1}.
\end{eqnarray*}
Next, set $y=\lambda+if'(t)$ and $\hat{w}(t,\lambda)=\tilde{w}(t,y)$, that is:
\begin{eqnarray*}
\hat{w}_t=\tilde{w}_t+if''(t)\tilde{w}_y,\qquad \hat{w}_\lambda=\tilde{w}_y
\end{eqnarray*}
which after substitution in (\ref{fou}) leads to
$$-\tilde{w}_t(t,y)+\frac{1}{2}(y-if'(t))^2\tilde{w}(t,y)=0.$$
Consequently
\begin{eqnarray*}
\tilde{w}(t,y)&=&\Pi(y)\exp\left\{-\frac{1}{2}\int_t^s(y-if'(u))^2du\right\}\\
&=&\Pi(y)\exp\left\{-\frac{1}{2}y^2(s-t)+iy\int_t^sf'(u)du+\frac{1}{2}\int_t^s(f'(u))^2du\right\}
\end{eqnarray*}
which alternatively implies that
\begin{eqnarray*}
w(t,a)&=&\frac{1}{2\pi}\int_{-\infty}^\infty \Pi(y)e^{-\frac{1}{2}y^2(s-t)+iy\int_t^sf'(u)du+\frac{1}{2}\int_t^s(f'(u))^2du}e^{i\lambda a}d\lambda\\
&=&\frac{1}{2\pi}\int_{-\infty}^\infty \Pi(y)e^{-\frac{1}{2}y^2(s-t)+iy\int_t^sf'(u)du+\frac{1}{2}\int_t^s(f'(u))^2du}e^{iy a+af'(t)}dy\\
&=&e^{\frac{1}{2}\int_t^s(f'(u))^2du+af'(t)}\frac{1}{2\pi}\int_{-\infty}^\infty \Pi(y)e^{-\frac{1}{2}y^2(s-t)+iy(a+\int_t^sf'(u)du)}dy
\end{eqnarray*}
as claimed.
\end{proof}

\begin{theorem} The density $\omega$ of the first passage time $T$ defined in (\ref{stop}),  is bounded by
\begin{eqnarray*}
&&h(s,a)e^{-af'(0)-\frac{1}{2}\int_0^s(f'(u))^2du-\int_0^sf''(u)\mathbb{E}^a(\tilde{X}_u)du}\\
&&\qquad \qquad\leq \omega\left(s, a+\int_0^sf'(u)du\right)\\
&&\qquad\qquad \qquad\qquad\leq h(s,a)e^{-af'(0)-\frac{1}{2}\int_0^s(f'(u))^2du}
\end{eqnarray*}
where $\omega$ is a solution of the forward heat equation and $h$ is as in (\ref{level}).
\end{theorem}
\begin{proof}
It follows from  Jensen's inequality,
\begin{eqnarray*}
\exp\left\{-\int_0^sf''(u)\tilde{\mathbb{E}}^{t,a}(\tilde{X}_u)du\right\}\leq\tilde{\mathbb{E}}^{t,a}[\exp\left\{-\int_0^sf''(u)\tilde{X}_udu\right\}]\leq 1
\end{eqnarray*}
and equations (\ref{schro}) and (\ref{sch1}).
\end{proof}



\begin{thebibliography}{xx}
\bibitem{kar} Karatzas, I. and S. Shreve. (1991). Brownian Motion and Stochastic Calculus, Springer-Verlag, New York.
\bibitem{pit} Pitman, J., and M. Yor (1998). The law of the maximum of a Bessel bridge. Technical Report No. 534.
\bibitem{tres} Revuz, D., and M. Yor. (2005). Continuous martingales and Brownian motion, Springer-Verlag, New York.
\end{thebibliography}
\end{document}